\documentclass{amsart}
\usepackage[utf8]{inputenc}
\usepackage[utf8]{inputenc}
\usepackage[T1]{fontenc}
\usepackage[all]{xy}
\usepackage{lipsum}
\usepackage{url}
\usepackage{stackrel}

\usepackage{amsmath,amsthm,amssymb}

\usepackage{graphicx}
\newtheorem{theorem}{Theorem}[section]
\newtheorem{lemma}[theorem]{Lemma}
\newtheorem{proposition}[theorem]{Proposition}
\theoremstyle{definition}
\newtheorem{definition}[theorem]{Definition}

\newtheorem{remark}[theorem]{Remark}
\newtheorem{corollary}[theorem]{Corollary}
\newtheorem{conjecture}[theorem]{Conjecture}
\numberwithin{equation}{section}

\begin{document}

%%%%%%%%%%%%%%%%%%%%%%%%%%%%%%%%%%%%%%%%%%%%%%%%%%%%%%%%%%%%
%%%%%%%%%%%%%%%%%%%%%%%%%%%%%%%%%%%%%%%%%%%%%%%%%%%%%%%%%%%%
% This a placeholder for the TOPLOGY PROCEEDINGS logo %%%%%%
%\noindent                                             %%%%%%
%\begin{picture}(150,36)                               %%%%%%
%\put(5,20){\tiny{Submitted to}}                       %%%%%%
%\put(5,7){\textbf{Topology Proceedings}}              %%%%%%
%\put(0,0){\framebox(140,34){}}                        %%%%%%
%\put(2,2){\framebox(136,30){}}                        %%%%%%
%\end{picture}                                        %%%%%%
%%%%%%%%%%%%%%%%%%%%%%%%%%%%%%%%%%%%%%%%%%%%%%%%%%%%%%%%%%%%
%%%%%%%%%%%%%%%%%%%%%%%%%%%%%%%%%%%%%%%%%%%%%%%%%%%%%%%%%%%%
\vspace{0.5in}

\renewcommand{\bf}{\bfseries}
\renewcommand{\sc}{\scshape}
%insert defs/styles
\vspace{0.5in}

\title[Non-contractible configuration spaces]%
{Non-contractible configuration spaces \\ }

%    Information for first author:
\author{Cesar A. Ipanaque Zapata}
\address{Deparatmento de Matem\'{a}tica,UNIVERSIDADE DE S\~{A}O PAULO
INSTITUTO DE CI\^{E}NCIAS MATEM\'{A}TICAS E DE COMPUTA\c{C}\~{A}O -
USP , Avenida Trabalhador S\~{a}o-carlense, 400 - Centro CEP:
13566-590 - S\~{a}o Carlos - SP, Brasil}
%    Current address (if needed):
%\curraddr{}
\email{cesarzapata@usp.br}

%    Information for second author (if needed):
%\author{Author Two}
%\address{}
%\email{}

%    General info
%%%%%%%%%%%%%%%%%%%%%%%%%%%%%%%%%%%%%%%%%%%%%%%%%%%
%                                                                                                                           %
%         Please use the current 2010 Mathematics Subject Classification:             %
%         http://www.ams.org/mathscinet/msc/                                                        %
%         http://www.zentralblatt-math.org/msc/en/                                                 %
%%%%%%%%%%%%%%%%%%%%%%%%%%%%%%%%%%%%%%%%%%%%%%%%%%%

\thanks {I am very grateful to Frederick Cohen and Jes\'{u}s Gonz\'{a}lez for their comments and encouraging remarks which were of invaluable mental support. Also, the author wishes to acknowledge support for this research, from FAPESP 2016/18714-8.}

\begin{abstract} Let $F(M,k)$ be the configuration space of ordered $k-$tuples of distinct points in the manifold $M$. Using the Fadell-Neuwirth fibration, we prove that
   the configuration spaces $F(M,k)$ are never contractible, for $k\geq 2$. As applications of our results, we will calculate the LS category and topological complexity for its loop space and suspension.
\end{abstract}

\maketitle

%%%%%%%%%%%%%%%%%%%%%%%%%%%%%%%%%%%%%%%%%%%%%%%%%%%%%%%%%%%%%%

\section{\bf Introduction}

Let $X$ be the space of all possible configurations or states of a mechanical system. A motion planning algorithm on $X$ is a function which assigns to any pair of configurations $(A,B)\in X\times X$, an initial state $A$ and a desired state $B$, a continuous motion of the system starting at the initial state $A$ and ending at the desired state $B$. The elementary problem of robotics, \textit{the motion planning problem}, consists of finding a motion planning algorithm for a given mechanical system. The motion planning algorithm should be continuous, that is, it depends continuously on the pair of points $(A,B)$. Absence of continuity will result in the instability of behavior of the motion planning. Unfortunately, a continuous motion planning on space $X$ exists if and only if $X$ is contractible, see \cite{farber2003topological}. The design of effective motion planning algorithms is one of the challenges of modern robotics, see, for example Latombe \cite{latombe2012robot} and LaValle \cite{lavalle2006planning}. 

Investigation of the problem of simultaneous motion planning without collisions for $k$ robots in a topological manifold $M$ leads one to study the (ordered) configuration space $F(M,k)$. We want to know if exists a continuous motion planning algorithm on the space $F(M,k)$. Thus, an interesting question is whether $F(M,k)$ is contractible.
 
% In \textit{the problem of simultaneous motion planning without collisions} for $k$ robots,

%In robotics, in \textit{the motion planning problem}, a continuous motion planning algorithm 
It seems likely that the configuration space $F(M,k)$ is not contractible for certain topological manifolds $M$. Evidence for this statement is given in the work of F. Cohen and S. Gitler, in  \cite{cohen2002loop}, they described the homology of loop spaces of the configuration space $F(M,k)$ whose results showed that this homology is non trivial. In a robotics setting, the (collision-free) motion planning problem is challenging since it is not known an effective motion planning algorithm, see \cite{le2018multi}. 

In this paper, using the Fadell-Neuwirth fibration, we will prove that the configuration spaces $F(M,k)$ of certain topological manifolds $M$, are never contractible (see Theorem \ref{theor4}). Note that the configuration space $F(X,k)$ can be contractible, for any $ k\geq 1$ (e.g. if $X$ is an infinite indiscrete space or if $X=\mathbb{R}^\infty$). As applications of our results, we will calculate the LS category and topological complexity for the (pointed) loop space $\Omega F(M,k)$ (see Theorem \ref{cat-and-tc-loop-config}) and the  suspension $\Sigma F(M,k)$ (see Theorem \ref{cat-suspensio-config} and Proposition \ref{tc-suspension-config}).
 
 %and F. Cohen \cite{cohen1998lusternik}
% A better version of the naive approach is the following
 
 \begin{conjecture}
 If $X$ is a path-connected and paracompact topological space with covering dimension $1\leq dim(X)<\infty$. Then the configuration spaces $F(X,k)$ are never contractible, for $k\geq 2$.
 \end{conjecture}

Computation of LS category and topological complexity of the configuration space $F(M,k)$ is a great challenge. The LS category of the configuration space  $F(\mathbb{R}^m,k)$ has been computed by Roth in \cite{roth2008category}. In Farber and Grant's work \cite{farber2009topological}, the authors computed the TC of the configuration space $F(\mathbb{R}^m,k)$. Farber, Grant and Yuzvinsky determined the topological complexity of $F(\mathbb{R}^m-Q_r,k)$ for $m=2,3$ in \cite{farber2007topological}. Later González and Grant extended the results to all dimensions $m$ in \cite{gonzalez2015sequential}. Cohen and Farber in \cite{cohen2011topological} computed the topological complexity of the configuration space $F(\Sigma_g-Q_r,k)$ of orientable surfaces $\Sigma_g$. Recently in \cite{zapata2017lusternik}, the author computed the LS category and TC of the configuration space $F(\mathbb{CP}^m,2)$. The LS category and TC of the configuration space of ordered $2-$tuples of distinct points in   $G\times\mathbb{R}^n$ has been computed by the author in \cite{zapata2017category}. Many more related results can be found in the recent survey papers \cite{cohen2018topological} and \cite{farber2017configuration}. 

%%%%%%%%%%%%%%%%%%%%%%%%%%%%%%%%%%%%%%%%%%%%%%%%%%%%%%%%%%%%%%%%%%%%%%%

\section{\bf Main Results}
Let $M$ denote a connected $m-$dimensional topological manifold (without boundary), $m\geq 1$. The \textit{configuration space} $F(M,k)$, of ordered $k-$tuples of distinct points in $M$ (see \cite{fadell1962configuration}) is the subspace of $M^k$ given by
 \[F(M,k)=\{(m_1,\ldots,m_k)\in M^k|~~m_i\neq m_j\text{ for all } i\neq j \}.\]
 
Let $Q_r=\{q_1,\ldots,q_r\}$ denote a set of $r$ distinct points of $M$. 
 
 Let $M$ be a connected finite dimensional topological manifold (without boundary) with dimension at least 2 and $k> r\geq 1$. It is well known that the projection map \begin{equation}
 \pi_{k,r}:F(M,k)\longrightarrow F(M,r),~(x_1,\ldots,x_k)\mapsto (x_1,\ldots,x_r)
 \end{equation} is a fibration with fibre $F(M-Q_r,k-r)$. It is called the Fadell-Neuwirth fibration \cite{fadell2001geometry}. In contrast, when the manifold $M$ has nonempty boundary, $\pi_{k,r}$ is not a fibration. The fact that the map $\pi_{k,r}$ is not a fibration may be seen by considering, for example, the manifold $M=\mathbb{D}^2$ that is with boundary but the fibre $\mathbb{D}^2-\{(0,0)\}$ is not homotopy equivalent to the fibre $\mathbb{D}^2-\{(1,0)\}$.  
 
 Let $X$ be a space, with base-point $x_0$. The pointed loop space is denoted by $\Omega X$, as its base-point, if it needs one, we take the function $w_0$ constant at $x_0$. We recall that a topological space $X$ is weak-contractible if all homotopy groups of $X$ are trivial, that is, $\pi_n(X,x_0)=0$ for all $n\geq 0$ and all choices of base point $x_0$.
 
 In this paper, using the Fadell-Neuwirth fibration, we prove the following theorem

\begin{theorem}\label{theor4}[Main Theorem]
 If $M$ is a connected finite dimensional topological 
 manifold, 
 then the configuration space $F(M,k)$ is not contractible (indeed, it is never weak-contractible), for any $ k\geq 2$.
\end{theorem}

\begin{remark}
Theorem \ref{theor4} can be proved using classifying spaces. I am very grateful to Prof. Nick Kuhn for his suggestion about the following proof. Let $M$ be a connected finite dimensional topological manifold. If the configuration space $F(M,k)$ was contractible, then the quotient $F(M,k)/S_k$ would be a finite dimensional model for the classifying space of the $k^{th}$ symmetric group $S_k$.  But if $G$ is a nontrivial finite group or even just contains any nontrivial elements of finite order, then there is no finite dimensional model for $BG$ because $H^\ast(G)$ is periodic. Thus $F(M,k)$ is never contractible for $k\geq 2$.
\end{remark}

 %%%%%%%%%%%%%%%%%%%%%%%%%%%%%%%%%%%%%%%%%%%%%%%%%%5

%%%%%%%%%%%%%%%%%%%%%%%%%%%%%%%%%%%%%%%%%%%%%%%%%%%%%%%%%%%%%%%%%%%%%%%%%%%%%%%%%%

\section{PROOF of Theorem \ref{theor4}}
% In this section we prove Theorem \ref{theor4}. 
The proof of Theorem \ref{theor4} is greatly simplified by actually working on two main steps:
\begin{itemize}
    \item[S1.] We first get the Theorem \ref{theor4} when $\pi_1(M)=0$ (Proposition \ref{cor2}).
    \item[S2.] Then we prove the Theorem \ref{theor4} when $\pi_1(M)\neq 0$ (It follows from Lemma \ref{epimor}).
\end{itemize}

Here we note that the manifolds being considered are without boundary.

% the simply-connected case (Proposition \ref{cor2}) and the non simply-connected case (It follows from Lemma \ref{epimor}). We begin by proving three lemmas needed for our proof. Here we note that the manifolds being considered are without boundary.

Step $S1$ above is accomplished proving the next four results.
 
 \begin{lemma} \label{theor2} %\label{lem3}
 Let $M$ denote a connected $m-$dimensional topological manifold, $m\geq 2$. If $r\geq 1$, then  the configuration space $F(M-Q_r,k)$ is not contractible (indeed, it is not weak-contractible), $\forall k\geq 2$.
 \end{lemma}
 \begin{proof}
 Recall that if $p:E\longrightarrow B$ is the projection map in a fibration with inclusion of the fibre $i:F\longrightarrow E$ such that $p$ supports a cross-section $\sigma$, then (1) $\pi_q(E)\cong \pi_q(F)\oplus \pi_q(B), ~\forall q\geq 2$ and (2) $\pi_1(E)\cong \pi_1(F)\rtimes \pi_1(B)$. 
 
 If $r\geq 1$, then the first coordinate projection map  $\pi:F(M-Q_r,k)\longrightarrow M-Q_r$ is a fibration with fibre $F(M-Q_{r+1},k-1)$ and $\pi$ admits a section  (\cite{fadell1962configuration}, Theorem 1). Thus (1) $\pi_q(F(M-Q_r,k))\cong \bigoplus_{i=0}^{k-1} \pi_q(M-Q_{r+i}),~\forall q\geq 2$ (\cite{fadell1962configuration}, Theorem 2) and (2) $\pi_1(F(M-Q_r,k))$ is isomorphic to \[( (\cdots (\pi_1(M-Q_{r+k-1})\rtimes\pi_1(M-Q_{r+k-2}))\cdots )\rtimes\pi_1(M-Q_{r+1}))\rtimes \pi_1(M-Q_r).\] 
 
 Finally, notice that $M-Q_{r+k-1}$ is homotopy equivalent to $\bigvee_{i=1}^{r+k-2}\mathbb{S}^{m-1} \vee (M-V)$, where $V$ is an open $m-$ball in $M$ such that $Q_{r+k-1}\subset V$ (\cite{husseini2002configuration}, Proposition 3.1). Thus $M-Q_{r+k-1}$ is not weak contractible, therefore $F(M-Q_r,k)$ is not weak-contractible.
 
 \end{proof}

 \begin{lemma}\label{homology-loop-space}
 If $M$ is a simply-connected finite dimensional topological manifold which is not weak-contractible, then the singular homology (with coefficients in a field $\mathbb{K}$) of $\Omega M$ does not vanish in sufficiently large degrees.
 \end{lemma}
 \begin{proof}
 By contradiction, we will suppose the singular homology of $\Omega M$ vanishes in sufficiently large degrees, that is, there exists an integer $q_0\geq 1$ such that, $H_q(\Omega M;\mathbb{K})=0,\forall q\geq q_0$, where $\mathbb{K}$ is a field. Let $f$ denote a nonzero homology class of maximal degree in $H_\ast(\Omega M;\mathbb{K})$. As $M$ is finite dimensional and not weak-contractible, let $b$ denote a nonzero homology class in $\widetilde{H}_\ast(M;\mathbb{K})$ of maximal degree (here $\widetilde{H}_\ast(-;\mathbb{K})$ denote reduced singular homology, with coefficients in a field $\mathbb{K}$). Notice that $b\otimes f$ survives to give a non-trivial class in the Serre spectral sequence abutting to $H_\ast(P(M,x_0);\mathbb{K})$, since $M$ is simply-connected, the local coefficient system $H_{\ast}(\Omega M;\mathbb{K})$ is trivial, where \begin{equation*}
 P(M,x_0)=\{\gamma\in PM\mid \gamma(0)=x_0\},
 \end{equation*} it is contractible. This is a contradiction and so the singular homology of $\Omega M$ does not vanish in sufficiently large degrees.  
 \end{proof}
 
 \begin{proposition}\label{theor33}
If $M$ is a simply-connected topological manifold which is not weak-contractible with dimension at least $2$, then the configuration space $F(M,k)$ is not contractible (indeed, it is never weak-contractible), $\forall k\geq 2$.
\end{proposition}
 \begin{proof}
 By hypothesis, $M$ is a  connected finite dimensional topological manifold of dimension at least $2$. Consequently, there is a fibration $F(M,k)\longrightarrow M$ with fibre $F(M-Q_1,k-1)$ ($k\geq 2$). We just have to note that in sufficiently large degrees, the singular homology, with coefficients in a field $\mathbb{K}$, of  $F(M-Q_1,k-1)$ vanishes, since $F(M-Q_1,k-1)$ is a connected finite dimensional topological manifold. 
 
 On the other hand, if $F(M,k)$ were weak-contractible, then the pointed loop space of $M$ is weakly homotopy equivalent to $F(M-Q_1,k-1)$ which it cannot be by Lemma \ref{homology-loop-space}. Thus, the configuration space $F(M,k)$ is not weak-contractible.
 \end{proof}
 
 \begin{proposition}\label{theor333}
If $M$ is a topological manifold which is weak-contractible with dimension at least $2$, then the configuration space $F(M,k)$ is not contractible (indeed, it is never weak-contractible), $\forall k\geq 2$.
\end{proposition}
\begin{proof}
By the homotopy long exact sequence of the fibration $F(M,k)\longrightarrow M$ with fibre $F(M-Q_1,k-1)$, we can conclude the inclusion $i:F(M-Q_1,k-1)\hookrightarrow F(M,k)$ is a weak homotopy equivalence. If $k\geq 3$, then Lemma \ref{theor2} implies that $F(M-Q_1,k-1)$ is not weak contractible and so $F(M,k)$ is not weak contractible. If $k=2$, we consider the cover $M=A\cup B,~A=M-\{q\}, B=M-\{q^\prime\}, q,q^\prime$ distinct. Here we note that $A=M-\{q\}$ and $B=M-\{q^\prime\}$ are homeomorphic to $M-Q_1$ and $A\cap B=M-\{q,q^\prime\}$ is not weak-contractible, because $M-\{q,q^\prime\}$ is homotopy equivalent to the wedge $\mathbb{S}^{m-1}\vee (M-V)$, where $V$ is an open $m-$ball in $M$ such that $\{q,q^\prime\}\subset V$ (\cite{husseini2002configuration}, Proposition 3.1). Thus, the Mayer-Vietoris  sequence, for the given cover, implies $M-Q_1$ is not weak contractible and so $F(M,2)$ is not weak contractible. Therefore, $F(M,k)$ is not weak-contractible.

\end{proof}

By Propositions \ref{theor33} and \ref{theor333} we have the following statement.

\begin{proposition}\label{cor2}
 If $M$ is a simply-connected topological manifold with dimension at least $2$, then the configuration space $F(M,k)$ is not contractible (indeed, it is never weak-contractible), $\forall k\geq 2$.
\end{proposition}

A key ingredient for step $S2$ is given by the next result.

\begin{lemma}\label{epimor}
 If $M$ is a connected finite dimensional topological manifold with dimension at least $2$, then the inclusion map $i:F(M,k)\longrightarrow M^k$ induces a homomorphism $i_\ast:\pi_1F(M,k)\longrightarrow \pi_1M^k$ which is surjective. 
 \end{lemma}
 \begin{proof}
 We will prove it by induction on $k$. We just have to note that the inclusion map $j:M-Q_k\longrightarrow M$ induces an epimorphism $j_\ast:\pi_1(M-Q_k)\longrightarrow \pi_1M$, for any $k\geq 1$. The following diagram of fibrations (see Figure \ref{lemadosquatroo})
 \begin{figure}[!ht]
\scalebox{1.0}{\includegraphics{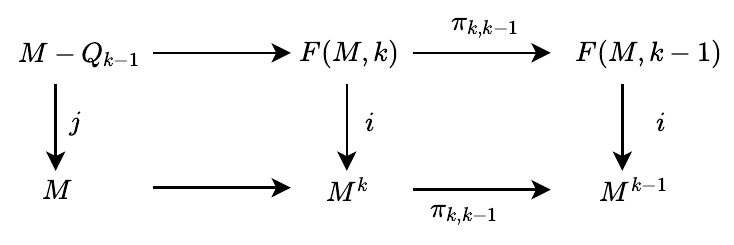}}
\caption{Commutative diagram.}
\label{lemadosquatroo}
\end{figure}
 is commutative. Thus by induction, we can conclude  the inclusion map $i:F(M,k)\longrightarrow M^k$ induces a homomorphism $i_\ast:\pi_1F(M,k)\longrightarrow \pi_1M^k$ which is surjective and so we are done.
 \end{proof}
 
 \begin{remark}
 Lemma \ref{epimor} is actually a very special case of a general theorem of Golasi{\'n}ski,  Gon{\c{c}}alves and Guaschi in (\cite{golasinski2017homotopy}, Theorem 3.2). Also, it can be proved using braids (\cite{goldberg1974exact}, Lemma 1).
 \end{remark}
 
 \noindent \textit{Proof of Theorem \ref{theor4}.} The case $dim~ M=1$ is straightforward, so we assume that  $dim~M\geq 2$. If $\pi_1(M)=0$ then the result follows easily from the Proposition  \ref{cor2}. If $\pi_1(M)\neq 0$ then $\pi_1(M^k)\neq 0$ and by Lemma \ref{epimor} 
 \[i_\ast:\pi_1(F(M,k))\longrightarrow \pi_1(M^k)\] is an epimorphism. Thus  $\pi_1(F(M,k))\neq 0$ and $F(M,k)$ is not weak contractible. Therefore, $F(M,k)$ is not contractible.
 \begin{flushright}
 $\square$
 \end{flushright}

%%%%%%%%%%%%%%%%%%%%%%%%%%%%%%%%%%%%%%%%%%%%%%%%%%%%%%%%%%%%%%%%%%%%%%%%%%%%%%%%%%%%%%% 
 \section{Lusternik-Schnirelmann category and topological complexity}
 As applications of our results, in this section, we will calculate the L-S category and topological complexity for the (pointed) loop space $\Omega F(M,k)$  and the  suspension $\Sigma F(M,k)$.
 
 Here we follow a definition of category, one greater than category given in  \cite{cornea2003lusternik}.
 
\begin{definition}\label{category-tc}
 We say that the Lusternik-Schnirelmann category or category of a topological space $X$, denoted $cat(X)$, is the least integer $m$ such that $X$ can be covered with $m$ open sets, which are all contractible within $X$. If no such $m$ exists we will set $cat(X)=\infty$.
\end{definition}

Let $PX$ denote the space of all continuous paths $\gamma: [0,1] \longrightarrow X$ in $X$ and  $\pi: PX \longrightarrow X \times X$ denotes the
map associating to any path $\gamma\in PX$ the pair of its initial and end points $\pi(\gamma)=(\gamma(0),\gamma(1))$. Equip the path space $PX$ with the compact-open topology. 

\begin{definition}\cite{farber2003topological}
  The \textit{topological complexity} of a path-connected space $X$, denoted by $TC(X)$, is the least integer $m$ such that the Cartesian product $X\times X$ can be covered with $m$ open subsets $U_i$, \begin{equation*}
        X \times X = U_1 \cup U_2 \cup\cdots \cup U_m 
    \end{equation*} such that for any $i = 1, 2, \ldots , m$ there exists a continuous function $s_i : U_i \longrightarrow PX$, $\pi\circ s_i = id$ over $U_i$. If no such $m$ exists we will set $TC(X)=\infty$. 
\end{definition}

\begin{remark}\label{basic-rem}
For all path connected spaces $X$, the basic inequality that relate $cat$ and $TC$ is\begin{equation}
cat(X)\leq TC(X).
\end{equation} On the other hand, by (\cite{farber2003topological}, Theorem 5), for all path connected paracompact spaces $X$, $TC(X)\leq 2cat(X)-1$. It follows from the Definition \ref{category-tc} that we have $cat(X)=1$ if and only if $X$ is contractible. It is also easy to show that  $TC(X)=1$ if and only if $X$ is contractible.
\end{remark}

By Remark \ref{basic-rem} and Theorem \ref{theor4}, we obtain the following statement.
\begin{proposition}
  If $M$ is a connected finite dimensional topological manifold, then the Lusternik-Schnirelmann category and the topological complexity of $F(M,k)$ are at least 2, $\forall k\geq 2$.
\end{proposition}

Proposition \ref{fred} and Lemma \ref{homotopy-finite-type} we state in this section are known, they can be found in the paper by Frederick R. Cohen \cite{cohen1998lusternik}. Here $\Omega^j_0X$ denotes the component of the constant map in the $j^{th}$ pointed loop space of $X$.

\begin{proposition}\label{fred}(\cite{cohen1998lusternik}, Theorem 1)
 If $X$ is a simply-connected finite complex which is not contractible, then the Lusternik-Schnirelmann category of $\Omega^j_0X$ is infinite for $j\geq 1$.
\end{proposition}

\begin{lemma}\label{homotopy-finite-type}
Let $M$ be a simply-connected finite dimensional topological manifold with dimension at least 3. If $M$ has the homotopy type of a finite CW complex, then the configuration space $F(M,k)$ has the homotopy type of a finite CW complex, $\forall k\geq 1$.
\end{lemma}

As a consequence of Theorem \ref{theor4} we can obtain Proposition \ref{fred} for configuration spaces.
%By (\cite{cohen1998lusternik}, Theorem 1) and Corollary \ref{cor2} we can conclude the following statement.

\begin{theorem}\label{cat-and-tc-loop-config}
Let $M$ be a space which has the homotopy type of a finite CW complex. If $M$ is a simply-connected finite dimensional topological manifold with dimension at least 3, then the Lusternik-Schnirelmann category and the topological complexity of $\Omega^j_0 F(M,k)$ are infinite, $\forall k\geq 2, ~j\geq 1$.
\end{theorem}
\begin{proof}
The assumptions that $M$ is a simply-connected finite dimensional topological manifold with dimension at least 3, imply the configuration space $F(M,k)$ is simply-connected. Furthermore, as $M$ has the homotopy type of a finite CW complex, the configuration space $F(M,k)$ also has the homotopy type of a finite CW complex by Lemma \ref{homotopy-finite-type}. Finally the configuration space $F(M,k)$ is not contractible by Theorem \ref{theor4}. Therefore we can apply Proposition \ref{fred} and conclude that the Lusternik-Schnirelmann category of $\Omega^j_0 F(M,k)$ is infinite, $\forall k\geq 2$. Moreover, by Remark \ref{basic-rem}, the topological complexity of $\Omega^j_0 F(M,k)$ is also infinite, $\forall k\geq 2$. 
\end{proof}

\begin{remark}
\begin{enumerate}
    \item In Theorem \ref{cat-and-tc-loop-config}, the assumption \textit{$M$ has the homotopy type of a finite CW complex} can be reduce to the assumption \textit{$M$ is a CW complex of finite type} (see \cite{wilkerson2006draft}). 
    \item By Theorem \ref{cat-and-tc-loop-config}, if $G$ is a simply-connected finite dimensional Lie group of finite type with dimension at least 3. Then the topological complexity $TC(\Omega F(G,k))=\infty,$ for any $k\geq 2$. In contrast, we will see that the topological complexity $TC(\Sigma F(G,k))=3<\infty,$ for any $k\geq 3$.
\end{enumerate}
\end{remark}

\begin{remark}
If $X$ is any topological space and $\Sigma X:=\dfrac{X\times [0,1]}{X\times \{0\}\cup X \times \{1\}}$ is the non-reduced suspension of the space $X$, it is well-known that $cat(\Sigma X)\leq 2$. We can cover $\Sigma X$ by two overlapping open sets (e.g, $q(X\times [0,3/4)$ and $q(X\times (1/4,1])$, where $q:X\times [0,1]\longrightarrow \Sigma X$ is the projection map), such that each open set is homeomorphic to the cone $CX:=\dfrac{X\times [0,1]}{X\times \{0\}}$, so is contractible in itself and thus it is contractible in the suspension $\Sigma X$. 
\end{remark}

\begin{lemma}\label{cat-suspension}
Let $X$ be a simply-connected topological space. If $X$ is not weak-contractible, then \begin{equation}
cat(\Sigma X)=2. 
\end{equation} 
\end{lemma} 
\begin{proof}
It is sufficient to prove that $\Sigma X$ is not weak-contractible and thus $cat(\Sigma X)\geq 2$. Since contractible implies weak-contractible. If $\Sigma X$ was weak-contractible then by the Mayer-Vietoris sequence for the open covering $\Sigma X=q(X\times [0,3/4)\cup q(X\times (1/4,1])$ we can conclude $H_q(X;\mathbb{Z})=0,\forall q\geq 1$. Thus by (\cite{hatcher2002algebraic}, Corollary 4.33) $X$ is weak-contractible (here we have used that $X$ is simply-connected\footnote{By Hatcher (\cite{hatcher2002algebraic}, Example 2.38) there exists nonsimply-connected acyclic spaces.}). It is a contradiction with the hypothesis. Therefore $\Sigma X$ is not weak-contractible.
\end{proof}

\begin{theorem}\label{cat-suspensio-config}
If $M$ is a simply-connected finite dimensional topological manifold with dimension at least 3, then \begin{equation} cat(\Sigma F(M,k))=2, \forall k\geq 2.
\end{equation}
\end{theorem}
\begin{proof}
The arguments $M$ is a simply-connected finite dimensional topological manifold with dimension at least 3, imply the configuration space $F(M,k)$ is simply-connected. The configuration space $F(M,k)$ is not weak-contractible by Theorem \ref{theor2}. Therefore we can apply Lemma \ref{cat-suspension} and the Lusternik-Schnirelmann category of $\Sigma F(M,k))$ is two, $\forall k\geq 2$.
\end{proof}

We note that $\Sigma F(M,k)$ is paracompact because $F(M,k)$ is paracompact. 
\begin{corollary}
  If $M$ is a simply-connected finite dimensional topological manifold with dimension at least 3, then \begin{equation} 2\leq TC(\Sigma F(M,k))\leq 3, \forall k\geq 2.
\end{equation}
\end{corollary}
\begin{proof}
It follows from Remark \ref{basic-rem} and Proposition \ref{cat-suspensio-config}. 
\end{proof}

\begin{remark}
By Corollary \ref{tc-suspension-config} the topological complexity of the suspension of a configuration space is secluded in the range $2\leq TC(\Sigma F(M,k))\leq 3$ and any value in between can be taken (e.g. if $M=\mathbb{S}^m$ or $\mathbb{R}^m$ and $k=2$).
\end{remark}

Now we will recall the definition of the cup-length.

\begin{definition}\cite{cornea2003lusternik}
  Let $R$ be a commutative ring with unit and $X$ be a topological space. The \textit{cup-length} of $X$, denote $cup_R(X)$, is the least integer $n$ such that all $(n+1)-$fold cup products vanish in the reduced cohomology $\widetilde{H^\star}(X;R)$.
\end{definition}

\begin{remark}(\cite{cornea2003lusternik}, Theorem 1.5)\label{cup-length}
Let $R$ be a commutative ring with unit and $X$ be a topological space. It is well-known that \[1+cup_R(X)\leq cat(X).\] 
\end{remark}

On the other hand, it is easy to verify that the cup-length has the property listed below.

\begin{lemma}\label{prop-cup-length}
Let $\mathbb{K}$ be a field and $X,Y$ be topological spaces. Then
if $H^k(Y;\mathbb{K})$ is a finite dimensional $\mathbb{K}-$vector space for all $k\geq 0$. We have \begin{equation*}cup_{\mathbb{K}}(X\times Y)\geq cup_{\mathbb{K}}(X)+cup_{\mathbb{K}}(Y).\end{equation*}
\end{lemma}

\begin{proposition}\label{tc-suspension-config}
%Let $G$ be a CW complex of finite type. 
If $G$ is a simply-connected finite dimensional Lie group of finite type with dimension at least 3. Then \begin{equation} TC(\Sigma F(G,k))= 3, \forall k\geq 3.
\end{equation} 
\end{proposition}
\begin{proof}
We will assume that $G$ is not contractible, the case $G$ is contractible follows easily because $F(G,k)$ is homotopy equivalent to $F(\mathbb{R}^d,k)$, where $d=dim(G)$ (see \cite{zapataespaccos}, pg. 118). By Corollary \ref{tc-suspension-config} it is sufficient to prove that $TC(\Sigma F(G,k))\neq 2$. 
If $TC(\Sigma F(G,k))=2$ then, by (\cite{grant2013spaces}, Theorem 1), we have  $\Sigma F(G,k)$ is homotopy equivalent to some (odd-dimensional) sphere. Then $F(G,k)$ is homotopy equivalent to some (even-dimensional) sphere and thus $cat(F(G,k))=2$. On the other hand, $F(G,k)$ is homeomorphic to the product $G\times F(G-\{e\},k-1)$ because $G$ is a topological group. Then $2=cat(G\times F(G-\{e\},k-1))\geq cup_{\mathbb{K}}(G\times F(G-\{e\},k-1))+1$ for any field $\mathbb{K}$ (see Remark \ref{cup-length}). Furthermore, Lemma \ref{prop-cup-length} implies that $cup_{\mathbb{K}}(G\times F(G-\{e\},k-1))\geq cup_{\mathbb{K}}(G)+cup_{\mathbb{K}}(F(G-\{e\},k-1))\geq 1+1=2$ (here we note that $k-1\geq 2$ and by Theorem \ref{theor4} we have the cup length $cup_{\mathbb{K}}(F(G-\{e\},k-1))\geq 1$). Thus, $2=cat(G\times F(G-\{e\},k-1))\geq 3$ which is a contradiction.

\end{proof}

%%%%%%%%%%%%%%%%%%%%%%%%%%%%%%%%%%%%%%%%%%%%%%%%%%%%%%%%%%%%%%%%%%%%%%%%%%%%%%%%%%%%%%%%%%%%%%%%%%55

\bibliographystyle{plain}

\end{document}